\documentclass{article}

\usepackage{amsfonts}
\usepackage{amsthm}
\usepackage{amsmath}
\usepackage{amssymb}
\usepackage{mathtools} 
\usepackage{authblk} 
\usepackage{xcolor}
\usepackage{soul} 
\usepackage{mathrsfs}

\usepackage[top=1in, bottom=1in, left=1in, right=1in]{geometry}
\newcommand{\norm}[1]{\left \lVert#1\right \rVert}
\newtheorem{thm}{Theorem}[section]
\newtheorem{lem}{Lemma}[section]

\newtheorem{assumption}{Assumption}[section]
\theoremstyle{definition}
\newtheorem{defn}{Definition}[section]

\theoremstyle{remark}

\definecolor{cucol}{rgb}{0,0,0.8}
\definecolor{afcol}{rgb}{1,0,0}

\numberwithin{equation}{section}

\begin{document}
	
	
	\title{Picard approximation of a singular backward stochastic nonlinear Volterra integral equation}
	
	\date{}
	

    \author[1,2]{Arzu Ahmadova \thanks{Corresponding author. Email: \texttt{arzu.ahmadova@uni-due.de}}}
	\author[2]{Nazim I. Mahmudov\thanks{ Email: \texttt{nazim.mahmudov@emu.edu.tr}}}
	\affil[1]{Faculty of Mathematics, University of Duisburg-Essen, 45127, Essen, Germany}
	
	\affil[2]{Department of Mathematics, Eastern Mediterranean University, Mersin 10, 99628, T.R. North Cyprus}
	
	
	\maketitle
	

	\begin{abstract}
		\noindent Backward stochastic differential equations (BSDEs) belong nowadays to the most frequently studied equations in stochastic analysis and computational stochastics. In this paper we prove that Picard iterations of BSDEs with globally Lipschitz continuous nonlinearities converge exponentially fast to the solution. Our main result in this paper is to establish a fundamental lemma to prove the global existence and uniqueness of an adapted solution to a singular backward stochastic nonlinear Volterra integral equation (for short singular BSVIE) of order $\alpha \in (\frac{1}{2},1)$ under a weaker condition than Lipschitz one in Hilbert space.
		
	\end{abstract}
	
	\section{Introduction}\label{Sec:intro}

The study of backward stochastic differential equations (BSDEs) has necessary applications in stochastic optimal control, stochastic differential games, the probabilistic formula for the solutions of quasilinear partial differential equations, and financial markets. The adapted solution for a linear BSDE arising as an adjoint process for a stochastic control problem was first studied by Bismut \cite{bismut} in 1973, then by Bensousssan \cite{bensoussan}, and while Pardoux and Peng \cite{pardoux-peng-1990} first studied the result for the existence and uniqueness of an adapted solution for a continuous general nonlinear BSDE, which is a final value problem for a stochastic differential equation of It\^{o} type under the uniform Lipschitz conditions of the following form:
	\begin{align*}
	\begin{cases*}
	\mathrm{d}Y(t)=h(t,Y(t),Z(t))\mathrm{d}t + Z(t)\mathrm{d}W(t), t \in [0,T],\\
	Y(T)=\xi.
	\end{cases*}
	\end{align*}
They proved the existence and uniqueness of an adapted solution by means of the Bihari's inequality, which is the most important generalization of the Gronwall-Bellman inequality.
Since then, the theory of BSDE became a powerful tool in many fields, such as financial mathematics, optimal control, semi-linear and quasi-linear partial differential equations.
Later, there have been many works devoted to the study of BSDEs and their applications in a series of papers \cite{pardoux-peng-1990,tang-li,pardoux-1999,peng-1993,peng-1992,hu,peng-1991,rong,tessitore,rong-1997,wang-yong} under the assumptions that the coefficients satisfy Lipschitz conditions. Moreover, Mao \cite{mao} obtained a more general result than that of Pardoux and Peng \cite{pardoux-peng-1990} in which he proved existence and uniqueness under mild assumptions by applying Bihari's inequality, which was the key tool in the proof.
	 
A few years later, Lin \cite{Lin} considered the following backward stochastic nonlinear Volterra integral equation.
	 \begin{equation*}
	 X(t)+\int_{t}^{T}f(t,s,X(s),Z(t,s))\mathrm{d}s+\int_{t}^{T}\left[ g(t,s,X(s))+Z(t,s)\right] \mathrm{d}W(s)=X.
	 \end{equation*}
	His goal in \cite{Lin} is to find a pair $\left\lbrace X(s),Z(t,s)\right\rbrace $ that requires that this pair $\left\lbrace \mathscr{F}_{t \vee s}\right\rbrace$-adapted and $Z(t,s)$ is related to $t$. This is the intersection point of our result on linear singular BSVIE with \cite{Lin} and differs from the case in \cite{pardoux-peng-1990,peng-1993,peng-1992,peng-1991,rong,mahmudov}. The author also defines $Z(t,s)=\tilde{Z}(t,s)-g(t,s)$, $(t,s)\in \mathcal{D}=\left\lbrace (t,s)\in \mathbb{R}^{2}_{+}; 0\leq t\leq s \leq T\right\rbrace $ as we defined for linear singular BSVIEs. Another intersection with \cite{Lin,yong,mahmudov,wang-yong} is the use of the well-known extended martingale representation theorem in which we consider the extended martingale representation to an adapted solution $\left\lbrace x(t), y(t,s) \right\rbrace $, $(t,s)\in \mathcal{D}$ also for linear singular BSVIEs.

Such types of equations introduced above have intersection points with our results when the following non-Lipschitz conditions are imposed on the functions $f$ and $g$:
 \begin{align*}
 &|f(t,x,y)-f(t,\bar{x},\bar{y})|^{2} \leq \kappa (|x-\bar{x}|^{2}) +c|y-\bar{y}|^{2},\quad \text{a.s.},\\
 &|g(t,x)-g(t,\bar{x})|^{2} \leq \kappa(|x-\bar{x}|^{2}), \text{a.s.},
 \end{align*}
 where $c>0$ and $\kappa$ is a concave increasing function from $\mathbb{R}_{+}$ to $\mathbb{R}_{+}$ such that $\kappa(0)=0$, $\kappa(u)>0$ for $u>0$ and 
 \begin{equation}
 \int_{0+}\frac{\mathrm{d}u}{\kappa(u)}=\infty.
 \end{equation}
 Since $\kappa$ is concave and $\kappa(0)=0$, there exist positive constants $a, b$ such that $\kappa(u)\leq a+bu$ for all $u\geq 0$.
 
 Another intersection in their work is an application of Bihari inequality to prove existence and uniqueness result of an adapted solution. Compared to their results, our result requires the following non-Lipschitz assumption for the functions $f$ and $g$:
  \begin{align*}
 &\|f(t,x,y)-f(t,\bar{x},\bar{y})\|^{2} \leq \rho (t,\|x-\bar{x}\|^{2}) +c\|y-\bar{y}\|^{2},\quad \text{a.s.},\\
 &\|g(t,x)-g(t,\bar{x})\|^{2} \leq \rho(t, \|x-\bar{x}\|^{2}), \text{a.s.},
 \end{align*}
	Now let us briefly introduce some notations used throughout the article. First, let us recall some spaces. Let $(\mathcal{H}, \|\cdot\|)$ and $(U, \|\cdot\|)$ be real separable Hilbert spaces with inner product $\langle \cdot, \cdot\rangle$. Let $\mathcal{L}(U,\mathcal{H})$ be the space of bounded linear operators mapping from $U$ to $\mathcal{H}$ and $(\Omega,\mathscr{F}, \mathbb{F}, \mathbb{P})$ with natural filtration $\mathbb{F} \coloneqq \left\lbrace \mathscr{F}_{t}\right\rbrace_{t \geq 0}$ satisfying \textit{usual conditions} is a complete probability space. $(w(t))_{t\geq 0} $ is a $\mathcal{Q}$-Wiener process on $(\Omega,\mathscr{F}, \mathbb{F}, \mathbb{P})$ with a linear covariance bounded operator $\mathcal{Q} \in \mathcal{L}(U)$ such that $\text{tr}\mathcal{Q}<\infty$. Furthermore, suppose that there exists a complete orthonormal system $\left\lbrace e_{k}\right\rbrace_{k\geq 1} $ in $U$, a bounded sequence of nonnegative real numbers $\lambda_{k}$, such that $\mathcal{Q}e_{k}=\lambda_{k}e_{k}$, $k=1,2,\ldots$ and a sequence $\left\lbrace \beta_{k}\right\rbrace _{k\geq 1}$of independent Brownian motions such that
	\begin{equation*}
		\langle w(t), e\rangle_{U}= \sum_{k=1}^{\infty}\sqrt{\lambda_{k}}\langle e_{k},e \rangle_{U}\beta_{k}(t), \quad e \in U, \quad t \geq 0.
	\end{equation*}
In addition, let $\mathcal{L}^{0}_{2}=\mathcal{L}_{2}(\mathcal{Q}^{1/2}U,\mathcal{H})$ be the space of Hilbert-Schmidt operators from $\mathcal{Q}^{1/2}U$ to $\mathcal{H}$ with the inner product $\|\varphi\|^{2}_{\mathcal{L}^{0}_{2}}= \text{tr}[\varphi \mathcal{Q} \varphi^{\ast}] <\infty$, $\varphi\in \mathcal{L}(U,\mathcal{H})$. We also consider that $\mathcal{L}_{2}^{\mathscr{F}_{t}}(\Omega,\mathcal{H})$ is the Hilbert space $\mathcal{H}$-valued, $\mathscr{F}_{t}$-measurable and square-integrable random variables $\xi$, i.e. $\textbf{E}\|\xi\|^{2}_{\mathcal{H}} <\infty$. Also, $\mathcal{L}_{2}(\Omega,\mathscr{F}_{t},\mathcal{H})$ is the Hilbert space of all $\mathscr{F}_{t}$-measurable square integrable variables with values in a Hilbert space $\mathcal{H}$. $\mathcal{L}_{2}^{\mathscr{F}_{t}}([0,T],\mathcal{H})$ is the Hilbert space of all
square integrable and $\mathscr{F}_{t}$-adapted processes with values in $\mathcal{H}$. We also consider that $\mathcal{L}_{2}^{\mathscr{F}_{t}}(\Omega,\mathcal{H})$ is the Hilbert space $\mathcal{H}$-valued, $\mathscr{F}_{t}$-measurable and square-integrable random variables $\xi$, i.e. $\textbf{E}\|\xi\|^{2}_{\mathcal{H}} <\infty$.
 
	To the best of our knowledge, we study singular backward stochastic nonlinear Volterra integral equation which is an unaddressed topic in the previous literature (singular BSVIE, for short) of order $\alpha\in (\frac{1}{2},1)$ on $[0,T]$ as follows:
	\begin{align}\label{fbsde}
	x(t)=\xi &+\int_{t}^{T}(s-t)^{\alpha-1}f(t,s,x(s),y(t,s))\mathrm{d}s\nonumber\\
	&+\int_{t}^{T}(s-t)^{\alpha-1}\left[g(t,s,x(s))+y(t,s) \right] \mathrm{d}w(s), \qquad \text{P-a.s}.
	\end{align}
where $f: [0,T]\times \mathcal{H} \times  \mathcal{L}^{0}_{2}\to \mathcal{H}$ and $g: \mathcal{D}\times \mathcal{H} \to \mathcal{L}^{0}_{2}$ are assumed to be measurable mappings and terminal value $\xi \in \mathcal{L}_{2}(\Omega, \mathscr{F}_{T},\mathcal{H})$ is an $\mathscr{F}_{T}$-measurable square integrable variables with values in $\mathcal{H}$ such that $\textbf{E}\|\xi\|^{2}< \infty$. Let $x(t,\omega)=x(t)$ be a stochastic process on $ [0,T]$ and $\omega\in \Omega$.
		The expectation operator $\textbf{E}$ is denoted by
		\begin{equation*}
		\left( \textbf{E}x\right)(t)\coloneqq \int_{\Omega}x(t)\mathbb{P}(\mathrm{d}\omega), \quad t \in [0,T].
		\end{equation*}
	 
	 Our goal in this paper is to search for a pair of stochastic processes $\left\lbrace x(t),y(t,s); (t,s)\in \mathcal{D}\right\rbrace $, which we require to be $\mathscr{F}_{t\vee s}$ -adapted and satisfy \eqref{fbsde} in the usual sense of It\^{o}. Such a pair is called an adapted solution of the equation \eqref{fbsde}. Our main result will be an existence and uniqueness result for a matched an adapted pair $\left\lbrace x(t),y(t,s); (t,s)\in \mathcal{D}\right\rbrace $ that solves \eqref{fbsde}. We first derive representations of the adapted solution and then study existence and uniqueness results under a weaker condition than the Lipschitz condition. Unlike other research discussed above, we apply the Carath\'{e}odory- type condition to prove the existence and uniqueness of the adapted solution of Eq. \eqref{fbsde}.

Hence the plan of this work is as follows. In section \ref{stochastic}, we establish a fundamental lemma that will play a key role in this paper. Section \ref{approximation} is devoted to the construction of a Picard-type approximation of the adapted solution to show existence and uniqueness under non-Lipschitz conditions using the Bihari inequality, and section \ref{open problems} is devoted to the conclusion.

To conclude the introductory section, we introduce the following definition, which will be used throughout this work.
	\begin{defn} 
		For any $t \in [0,T]$, we define 	$M[t,T]$  to be a Banach space
		\begin{equation*}
		M[t,T]\coloneqq \mathcal{L}_{2}^{\mathscr{F}}(\Omega,C([t,T],\mathcal{H}))\times \mathcal{L}_{2}^{\mathscr{F}}(\mathcal{D},\mathcal{L}_{2}^{0})
		\end{equation*}
		 endowed with the norm
		 \begin{equation*}
		 \|(x,y)\|^{2}_{t}=\textbf{E}\sup_{t\leq s\leq T}\|x(s)\|^{2}+\textbf{E}\int_{t}^{T}\int_{s}^{T}\|y(s,u)\|^{2}\mathrm{d}u\mathrm{d}s< \infty.
		 \end{equation*}
	\end{defn}

	

	

\section{Fundamental lemma}\label{stochastic}

In this section, we establish a fundamental lemma to prove existence and uniqueness result using Picard type iteration in Section \ref{approximation}.

\begin{defn}
	A pair of adapted process $(x,y)\in M[t,T]$ is a mild solution of \eqref{fbsde} for all $t \in [0,T]$ if satisfies the backward stochastic nonlinear Volterra integral equation \eqref{fbsde}.
\end{defn}

We now introduce a fundamental lemma which plays an efficient role throughout this paper. To do so, we consider backward linear stochastic Volterra integral equation.

\begin{lem}\label{lem1}
For any $(x,y)\in M[t,T]$, the linear singular BSVIE
\begin{align}\label{3}
x(t)=\xi &+\int_{t}^{T}(s-t)^{\alpha-1}f(t,s)\mathrm{d}s\nonumber\\
&+\int_{t}^{T}(s-t)^{\alpha-1}\left[g(t,s)+y(t,s) \right] \mathrm{d}w(s), \quad \text{P-a.s.}
\end{align}
admits unique pair in $M[0,T]$ and moreover
\begin{align}\label{ineq1}
\textbf{E}\sup_{t\leq s\leq T}\|x(s)\|^{2}+\textbf{E} \int_{t}^{T}\int_{s}^{T}\|y(s,u)\|^{2}\mathrm{d}u\mathrm{d}s&\leq 8\textbf{E}\|\xi\|^{2}+16T\textbf{E}\|\xi\|^{2}\nonumber\\
&+ \frac{16(2T)^{2\alpha}}{2\alpha-1} \textbf{E}\int_{t}^{T}\|f(t,r)\|^{2}\mathrm{d}r\nonumber\\
&+2\textbf{E} \int_{t}^{T}\int_{s}^{T}\|g(s,u)\|^{2}\mathrm{d}u\mathrm{d}s.
\end{align}
\end{lem}
\begin{proof}
	
\textit{Uniqueness:} Let $(x_{1},y_{1})$ and $(x_{2},y_{2})$ be two solutions of \eqref{3}.
\begin{align*}
x_{1}(t)-x_{2}(t)=\int_{t}^{T}(s-t)^{\alpha-1}\left[y_{1}(t,s)-y_{2}(t,s) \right] \mathrm{d}w(s),
\end{align*}
 Taking $\textbf{E}\left\lbrace \cdot \mid \mathscr{F}_{t} \right\rbrace $ from above, we can deduce that
 \begin{align*}
 \textbf{E}\left\lbrace x_{1}(t)-x_{2}(t)\mid \mathscr{F}_{t} \right\rbrace =0, \quad \forall t\in [0,T],
 \end{align*}
It is obvious that $x_{1}(t)=x_{2}(t)$ and this follows that $y_{1}(t,s)=y_{2}(t,s)$.

\textit{Existence:} Taking a conditional expectation from \eqref{3}, we have 

\begin{equation*}
x(t)=\textbf{E}\left\lbrace \xi \mid \mathscr{F}_{t}\right\rbrace+\int_{t}^{T}(s-t)^{\alpha-1}\textbf{E}\left\lbrace f(t,s) \mid \mathscr{F}_{t}\right\rbrace\mathrm{d}s.
\end{equation*}
From extended martingale representation theorem, there exists $L(\cdot)\in \mathcal{L}_{2}^{\mathscr{F}}([0,T],\mathcal{L}^{0}_{2})$ and uniquely $K(t,\cdot)\in  \mathcal{L}_{2}^{\mathscr{F}}(\mathcal{D};\mathcal{L}^{0}_{2})$ which satisfy the following relations:

\begin{equation}\label{L}
\textbf{E}\left\lbrace \xi \mid \mathscr{F}_{t}\right\rbrace=\textbf{E}\xi+\int_{0}^{t}L(u)\mathrm{d}w(u),
\end{equation}
\begin{equation}\label{K}
\textbf{E}\left\lbrace f(t,s) \mid  \mathscr{F}_{t}\right\rbrace=\textbf{E}f(t,s)+\int_{0}^{t}K(s,u)\mathrm{d}w(u).
\end{equation}
Note also from \eqref{K}, we can easily deduce that $\forall s \in [0,T]$ 
\begin{equation*}
K(s,u)=0, \quad \text{a.e.}, \quad u\in [s,T], \quad \text{a.s.}
\end{equation*}
 and that
\begin{equation}
\textbf{E}\int_{0}^{T}\int_{0}^{s}|K(s,u)|^{2}\mathrm{d}u\mathrm{d}s\leq 4\textbf{E}\int_{0}^{T}|f(t,s)|^{2}\mathrm{d}s.
\end{equation}
Since $t\in [0.T]$, it is obvious that
\begin{align*}
\xi&=\textbf{E}\xi +\int_{0}^{T}L(u)\mathrm{d}w(u)\\
&=\textbf{E}\xi +\int_{0}^{t}L(u)\mathrm{d}w(u)+\int_{t}^{T}L(u)\mathrm{d}w(u)\\
&=\textbf{E}\left\lbrace \xi \mid \mathscr{F}_{t}\right\rbrace +\int_{t}^{T}L(u)\mathrm{d}w(u),
\end{align*}
and since $s\geq t$, we have 
\begin{align*}
f(t,s)&=\textbf{E}f(t,s)+\int_{0}^{s}K(s,u)\mathrm{d}w(u)\\
&=\textbf{E}f(t,s)+\int_{0}^{t}K(s,u)\mathrm{d}w(u)+\int_{t}^{s}K(s,u)\mathrm{d}w(u)\\
&=\textbf{E}\left\lbrace f(t,s) \mid  \mathscr{F}_{t}\right\rbrace+\int_{t}^{s}K(s,u)\mathrm{d}w(u).
\end{align*}
Therefore, we obtain
\begin{equation}\label{ll}
\textbf{E}\left\lbrace \xi \mid  \mathscr{F}_{t}\right\rbrace=\xi-\int_{t}^{T}L(u)\mathrm{d}w(u),
\end{equation}
and 
\begin{equation}\label{kk}
\textbf{E}\left\lbrace f(t,s) \mid  \mathscr{F}_{t}\right\rbrace= f(t,s)-\int_{t}^{s}K(s,u)\mathrm{d}w(u).
\end{equation}
Substituting \eqref{ll} and \eqref{kk} into \eqref{4} and using stochastic Fubini's theorem, we have 
\allowdisplaybreaks
\begin{align*}
x(t)&=\left(\xi-\int_{t}^{T}L(u)\mathrm{d}w(u) \right)+\int_{t}^{T}(s-t)^{\alpha-1} \left(f(t,s)-\int_{t}^{s}K(s,u)\mathrm{d}w(u) \right)\mathrm{d}s \\
&=\xi +\int_{t}^{T}(s-t)^{\alpha-1} f(t,s)\mathrm{d}s\\
&-\int_{t}^{T}L(u)\mathrm{d}w(u)-\int_{t}^{T}(s-t)^{\alpha-1}\int_{t}^{s}K(s,u)\mathrm{d}w(u)\mathrm{d}s\\
&=\xi +\int_{t}^{T}(s-t)^{\alpha-1} f(t,s)\mathrm{d}s\\
&-\int_{t}^{T}L(u)\mathrm{d}w(u)-\int_{t}^{T}\int_{u}^{T}(s-t)^{\alpha-1}K(s,u)\mathrm{d}s\mathrm{d}w(u).
\end{align*}
Thus, we get
\begin{align*}
x(t)=\xi +\int_{t}^{T}(s-t)^{\alpha-1} f(t,s)\mathrm{d}s+\int_{t}^{T}\tilde{y}(t,u)\mathrm{d}w(u).
\end{align*}
Then there exists a mild solution $(x,y)\in M[0,T]$ of \eqref{3} given by
	 \begin{equation}\label{4}
	 x(t)=\textbf{E}\left\lbrace \xi \mid \mathscr{F}_{t}\right\rbrace+\int_{t}^{T}(s-t)^{\alpha-1}\textbf{E}\left\lbrace f(t,s) \mid \mathscr{F}_{t}\right\rbrace\mathrm{d}s,
	 \end{equation}
	 and 
	\begin{equation}\label{5}
	\tilde{y}(t,u)=-L(u)-\int_{u}^{T}(s-t)^{\alpha-1}K(s,u)\mathrm{d}s.
	\end{equation} 
	We finally define $y(t,u)=\tilde{y}(t,u)-g(t,u)$, $(t,u)\in \mathcal{D}=\left\lbrace (t,u)\in \mathbb{R}^{2}_{+};0\leq t\leq u\leq T\right\rbrace $.
It is easily seen that the pair $(x,y)$ solves \eqref{3}. Therefore, the existence is proved.

From \eqref{ll} and \eqref{kk}, we invoke the following inequalities for $0\leq t\leq s\leq T$:
\begin{equation*}
\textbf{E}\int_{t}^{T}\|L(u)\|^{2}\mathrm{d}u\leq 4\textbf{E}\|\xi\|^{2},
\end{equation*}
and
\begin{equation*}
\textbf{E}\int_{t}^{s}\|K(s,u)\|^{2}\mathrm{d}u\leq 4\textbf{E}\|f(t,s)\|^{2}.
\end{equation*}	 
Now we estimate the solution $(x,y)$ given by  \eqref{4} and \eqref{5} in $[0,T]$. From \eqref{4} it follows that
	 \begin{align*}
	 \textbf{E}\sup_{t\leq s\leq T}\|x(s)\|^{2}&\leq 2\textbf{E}\sup_{t\leq s\leq T}\|\textbf{E}\left\lbrace \xi | \mathscr{F}_{s} \right\rbrace\|^{2}\\
	 &+2\textbf{E}\sup_{t\leq s\leq T}\Big(\int_{s}^{T} (r-s)^{\alpha-1}\textbf{E}\left\lbrace \|f(t,r)\| \quad| \mathscr{F}_{s}\right\rbrace \mathrm{d}r\Big)^{2}\coloneqq \mathcal{I}_{1}+\mathcal{I}_{2}.
	 \end{align*}
	 From Doob's inequality and the law of total expectation it follows that
	 \begin{align*}
	 \mathcal{I}_{1}\leq 2\textbf{E}\sup_{t\leq s\leq T}\textbf{E}\|\left\lbrace \xi | \mathscr{F}_{s} \right\rbrace\|^{2}\leq 8\textbf{E}\left( \textbf{E}\|\left\lbrace \xi | \mathscr{F}_{t} \right\rbrace\|^{2}\right) \leq 8\textbf{E}\| \xi\|^{2}.
	 \end{align*}
	 Doob's inequality and Jensen's inequality  in probabilistic setting imply that
	 \begin{align*}
	 \mathcal{I}_{2}&\leq 2\textbf{E}\sup_{t\leq s\leq T}\left(\textbf{E}\Biggl\{ \int_{s}^{T}(r-s)^{\alpha-1}\|f(t,r)\|\mathrm{d}r  \quad| \mathscr{F}_{s} \Biggr\} \right)^{2}\\
	 &\leq 2\textbf{E}\sup_{t\leq s\leq T}\left(\textbf{E}\Biggl\{\sup_{t\leq \tau\leq T} \int_{\tau}^{T}(r-\tau)^{\alpha-1}\|f(t,r)\|\mathrm{d}r  \quad| \mathscr{F}_{s} \Biggr\} \right)^{2}\\
	 &\leq 8\textbf{E}\left(\sup_{t\leq \tau\leq T} \int_{\tau}^{T}(r-\tau)^{\alpha-1}\|f(t,r)\|\mathrm{d}r \right)^{2}\\
	 &\leq 8\textbf{E}\sup_{t\leq \tau\leq T} \int_{\tau}^{T}(r-\tau)^{2\alpha-2}\mathrm{d}r \int_{\tau}^{T}\|f(t,r)\|^{2}\mathrm{d}r \\
	&\leq 8\frac{(T-t)^{2\alpha-1}}{2\alpha-1}\ \textbf{E}\int_{t}^{T}\|f(t,r)\|^{2}\mathrm{d}r\\
    &\leq 8\frac{(2T)^{2\alpha}}{2\alpha-1} \textbf{E}\int_{t}^{T}\|f(t,r)\|^{2}\mathrm{d}r.
	 \end{align*}
	 Eventually, we have
	 \begin{align}\label{6}
	  \textbf{E}\sup_{t\leq s\leq T}\|x(s)\|^{2}\leq 8\textbf{E}\| \xi\|^{2}+\frac{8(2T)^{2\alpha}}{2\alpha-1}\textbf{E}\int_{t}^{T}\|f(t,r)\|^{2}\mathrm{d}r.
	 \end{align}

Next we estimate $\tilde{y}$ using H\"{o}lder's inequality. We attain,

\begin{align*}
\|\tilde{y}(s,u)\|^{2}&\leq 2\|L(u))\|^{2}+2\norm{\int_{u}^{T}(r-s)^{\alpha-1}K(r,u)\mathrm{d}r}^{2}\\
&\leq  2\|L(u)\|^{2}+2\left(\frac{(T-s)^{2\alpha-1}}{2\alpha-1}-\frac{(u-s)^{2\alpha-1}}{2\alpha-1} \right) \int_{u}^{T}\|K(r,u)\|^{2}\mathrm{d}r\\
&\leq  2\|L(u)\|^{2}+2\frac{(T+u)^{2\alpha-1}}{2\alpha-1} \int_{u}^{T}\|K(r,u)\|^{2}\mathrm{d}r\\
&\leq 2\|L(u)\|^{2}+2\frac{(2T)^{2\alpha-1}}{2\alpha-1} \int_{u}^{T}\|K(r,u)\|^{2}\mathrm{d}r.
\end{align*}

Taking double integral of above inequality and applying Fubini's theorem twice yield that 
\allowdisplaybreaks
\begin{align*}
&\textbf{E}\sup_{t \leq \tau \leq T} \int_{\tau}^{T}\int_{s}^{T}\|\tilde{y}(s,u)\|^{2}\mathrm{d}u\mathrm{d}s \leq 2\textbf{E}\sup_{t \leq \tau \leq T}\int_{\tau}^{T}\int_{s}^{T}\|L(u)\|^{2}\mathrm{d}u\mathrm{d}s\nonumber\\
&+2\frac{(2T)^{2\alpha-1}}{2\alpha-1} \textbf{E}\sup_{t \leq \tau \leq T}\int_{\tau}^{T}\int_{s}^{T}\int_{u}^{T}\|K(r,u)\|^{2}\mathrm{d}r\mathrm{d}u\mathrm{d}s\nonumber\\
&\leq 8(T-t)\textbf{E}\|\xi\|^{2}+2\frac{(2T)^{2\alpha-1}}{2\alpha-1} \textbf{E}\sup_{t \leq \tau \leq T}\int_{\tau}^{T}\int_{s}^{T}\int_{s}^{r}\|K(r,u)\|^{2}\mathrm{d}u\mathrm{d}r\mathrm{d}s\nonumber\\
&\leq 8T\textbf{E}\|\xi\|^{2}+8\frac{(2T)^{2\alpha-1}}{2\alpha-1}\textbf{E}\sup_{t \leq \tau \leq T}\int_{\tau}^{T}\int_{s}^{T}\|f(t,r)\|^{2}\mathrm{d}r \mathrm{d}s\nonumber\\
&\leq 8T\textbf{E}\|\xi\|^{2}+8\frac{(2T)^{2\alpha-1}}{2\alpha-1}\textbf{E}\sup_{t \leq \tau \leq T}\int_{\tau}^{T}\int_{\tau}^{r}\|f(t,r)\|^{2}\mathrm{d}s \mathrm{d}r\nonumber\\
&\leq 8T\textbf{E}\|\xi\|^{2}+8\frac{(2T)^{2\alpha-1}}{2\alpha-1}\textbf{E}\sup_{t \leq \tau \leq T}\int_{\tau}^{T}\int_{\tau}^{r}\|f(t,r)\|^{2} \mathrm{d}r\\
&\leq 8T\textbf{E}\|\xi\|^{2}+8\frac{(2T)^{2\alpha-1}(T-t)}{2\alpha-1} \textbf{E}\int_{t}^{T}\|f(t,r)\|^{2}\mathrm{d}r\\
&\leq 8T\textbf{E}\|\xi\|^{2}+8\frac{(2T)^{2\alpha-1}T}{2\alpha-1} \textbf{E}\int_{t}^{T}\|f(t,r)\|^{2}\mathrm{d}r\\
&\leq 8T\textbf{E}\|\xi\|^{2}+\frac{4(2T)^{2\alpha}}{2\alpha-1} \textbf{E}\int_{t}^{T}\|f(t,r)\|^{2}\mathrm{d}r.
\end{align*}
Thus, we get
\begin{align}\label{7}
\textbf{E} \int_{t}^{T}\int_{s}^{T}\|\tilde{y}(s,u)\|^{2}\mathrm{d}u\mathrm{d}s \leq 8T\textbf{E}\|\xi\|^{2}+\frac{4(2T)^{2\alpha}}{2\alpha-1} \textbf{E}\int_{t}^{T}\|f(t,r)\|^{2}\mathrm{d}r.
\end{align}
Since we know $y(t,u)=\tilde{y}(t,u)-g(t,u)$, we also have
\begin{equation}\label{norm1}
	\int_{t}^{T}\int_{s}^{T}\|y(s,u)\|^{2}\mathrm{d}u\mathrm{d}s\leq 2\int_{t}^{T}\int_{s}^{T}\|\tilde{y}(s,u)\|^{2}\mathrm{d}u\mathrm{d}s+2\int_{t}^{T}\int_{s}^{T}\|g(s,u)\|^{2}\mathrm{d}u\mathrm{d}s.
\end{equation}
Taking into account \eqref{norm1} and summing over  \eqref{6} and \eqref{7} yield that
\begin{align*}
\textbf{E}\sup_{t\leq s\leq T}\|x(s)\|^{2}+\textbf{E} \int_{t}^{T}\int_{s}^{T}\|y(s,u)\|^{2}\mathrm{d}u\mathrm{d}s&\leq 8\textbf{E}\|\xi\|^{2}+16T\textbf{E}\|\xi\|^{2}\\
&+\frac{16(2T)^{2\alpha}}{2\alpha-1} \textbf{E}\int_{t}^{T}\|f(t,r)\|^{2}\mathrm{d}r\\
&+2\textbf{E} \int_{t}^{T}\int_{s}^{T}\|g(s,u)\|^{2}\mathrm{d}u\mathrm{d}s.
\end{align*}
Therefore, the proof is complete.
\end{proof}

\section{Picard approximation}\label{approximation}

In this section, we introduce the existence and uniqueness problem of the solution to fractional backward stochastic evolution equation in more general form, that is, if the function $f(t,x,y)$ is a non-Lipschitzian function. This can be constructed by an approximate sequence using the Picard type iteration. 
Let $\left\lbrace x_{j}, y_{j}\right\rbrace $ be a sequence in $M[0,T]$ defined recursively by

\begin{align}\label{9}
\begin{cases}
(x_{0}(t),y_{0}(t,s))=(x(T),0)=(\xi,0)\\
x_{j}(t)=\xi+\int_{t}^{T}(s-t)^{\alpha-1}f(s,x_{j-1}(s),y_{j}(t,s))\mathrm{d}s\\
\hspace{1.2cm}+\int_{t}^{T}(s-t)^{\alpha-1}\left[g(t,s,x_{j-1}(s))+y_{j}(t,s) \right] \mathrm{d}w(s), j\geq 1.
\end{cases}
\end{align}
To state our main results, we impose the following assumptions on the functions $f$ and $g$.

\begin{assumption}\label{H1}
	$f(\cdot,\cdot,0,0) \in \mathcal{L}_{2}(0,T;\mathcal{H})$ and $g(\cdot,\cdot,0) \in \mathcal{L}_{2}(0,T;\mathcal{L}^{0}_{2})$.
\end{assumption}
\begin{assumption}\label{H11}
	Let $\varpi\coloneqq 1- 8\frac{(2T)^{2\alpha}}{2\alpha-1}c>0$.
\end{assumption}
\begin{assumption}\label{H3}
	For all $x,\bar{x} \in \mathcal{H}$, $y,\bar{y}\in \mathcal{L}^{0}_{2}$ and $0\leq t\leq T$,
	\begin{align*}
    &\|f(t,s,x,y)-f(t,s,\bar{x},\bar{y})\|^{2} \leq \rho (\|x-\bar{x}\|^{2}) +c\|y-\bar{y}\|^{2},\quad \text{a.s.},\\
	&\|g(t,s,x)-g(t,s,\bar{x})\|^{2}\leq \rho (\|x-\bar{x}\|^{2}),
	\end{align*}
	where  $\rho(u)$ satisfies :
	
 	\begin{itemize}
		\item 	$\rho(\cdot)$ is a concave nondecreasing function from $\mathbb{R}_{+}$ to $\mathbb{R}_{+}$  such that $\rho(0)=0$, $\rho(u)>0$ for $u>0$ and
		\begin{equation*}
		\int_{0+}\frac{du}{\rho(u)}=\infty.
		\end{equation*}
		\item there exists $a\geq 0$, $b\geq 0$ such that  
		\begin{equation*}
		\rho(u)\leq a+bu,
		\end{equation*}
		for all $u\geq 0$;
	\end{itemize} 
\end{assumption}

Now we introduce some important constants which are used throughout this work.

\begin{align}\label{constants}
C_{1}&=16\frac{(2T)^{2\alpha}}{2\alpha-1} \textbf{E}\int_{t}^{T}\int_{s}^{T}\left(2\|f(s,u,0,0)\|^{2}+2a\right)\mathrm{d}u \mathrm{d}s+2\textbf{E} \int_{t}^{T}\int_{s}^{T}\left(2a+2\|g(s,u,0)\|^{2}\right) \mathrm{d}u\mathrm{d}s, \nonumber\\
C_{2}&=4bT\left(1+ 36\frac{T^{2\alpha}}{2\alpha-1}\right) ,\nonumber\\
C_{3}&=\left( 16 \frac{(2T)^{2\alpha}}{2\alpha-1}+2(T-t)\right),\nonumber\\
C_{4}&=C_{3}\rho\left(4C_{1}\exp(C_{2}T)\right).
\end{align}
\begin{lem}\label{lem2}
	Under Assumptions \ref{H1} and \ref{H3}, for all $t\in[0,T]$ and $j\geq 1$.

	\begin{equation}\label{11}
	\textbf{E}\left( \sup_{t\leq s\leq T}\|x_{j}(s)\|^{2}\right)\leq C_{1}\exp(C_{2}(T-t)),
	\end{equation}
	\begin{equation}\label{12}
	\textbf{E} \int_{t}^{T}\int_{s}^{T}\|y_{j}(s,u)\|^{2}\mathrm{d}u\mathrm{d}s\leq  \varpi^{-1} C_{1}\left(1+C_{2}(T-t)\exp(C_{2}(T-t)) \right). 
	\end{equation}
\end{lem}

\begin{proof}
	It follows from Lemma \ref{lem1} that
	\begin{align}\label{10}
    \textbf{E}\sup_{t\leq s\leq T}\|x(s)\|^{2}+\textbf{E} \int_{t}^{T}\int_{s}^{T}\|y_{j}(s,u)\|^{2}\mathrm{d}u\mathrm{d}s&\leq 8\textbf{E}\|\xi\|^{2}+16T\textbf{E}\|\xi\|^{2}\nonumber\\
    &+16\frac{(2T)^{2\alpha}}{2\alpha-1} \textbf{E}\int_{t}^{T}\int_{r}^{T}\|f(r,u,x_{j-1}(u),y_{j}(r,u))\|^{2}\mathrm{d}u\mathrm{d}r\nonumber\\
    &+2\textbf{E} \int_{t}^{T}\int_{r}^{T}\|g(r,u,x_{j-1}(u))\|^{2}\mathrm{d}u\mathrm{d}r.
	\end{align}
	Using Assumptions \ref{H1} and \ref{H3}, we have 
	\begin{align*}
	\|f(t,s,x_{j-1}(s),y_{j}(s,u))\|^{2}&=\|f(t,s,x_{j-1}(s),y_{j}(s,u))-f(t,s,0,0)+f(t,s,0,0)\|^{2}\\
	&\leq  2\|f(t,s,0,0)\|^{2}+2a+2b\|x_{j-1}(s)\|^{2}+2c\|y_{j}(s,u)\|^{2}\\
	\|g(t,s,x_{j-1}(s))\|^{2}&=\|g(t,s,x_{j-1}(s))-g(t,s,0)+g(t,s,0)\|^{2}\\
	&\leq 2a+2\|g(t,s,0)\|^{2}+2b\|x_{j-1}(s)\|^{2}
	\end{align*}

Substituting these into \eqref{10} yields that
	\begin{align*}
	\textbf{E}\sup_{t\leq s\leq T}\|x_{j}(s)\|^{2}&+\textbf{E} \int_{t}^{T}\int_{s}^{T}\|y_{j}(s,u)\|^{2}\mathrm{d}u\mathrm{d}s\leq 	8\textbf{E}\|\xi\|^{2}+16T\textbf{E}\|\xi\|^{2}\nonumber\\
	&+16 \frac{(2T)^{2\alpha}}{2\alpha-1} \textbf{E}\int_{t}^{T}\int_{s}^{T}\left(2\|f(s,u,0,0)\|^{2}+2a+2b\|x_{j-1}(u)\|^{2}+2c\|y_{j}(s,u)\|^{2} \right) \mathrm{d}u\mathrm{d}s\nonumber\\
	&+2\textbf{E} \int_{t}^{T}\int_{s}^{T}\left(2a+2\|g(s,u,0)\|^{2}+2b\|x_{j-1}(u)\|^{2} \right) \mathrm{d}u\mathrm{d}s.
    \end{align*}

Thus, we get
\begin{align}\label{13}
\textbf{E}\sup_{t\leq s\leq T}\|x_{j}(s)\|^{2}&+\left(1-32\frac{(2T)^{2\alpha}}{2\alpha-1}c \right) \textbf{E} \int_{t}^{T}\int_{s}^{T}\|y_{j}(s,u)\|^{2}\mathrm{d}u\mathrm{d}s\nonumber\\
&\leq C_{1}+C_{2}\textbf{E}\int_{t}^{T}\sup_{s \leq r \leq T}\left( \|x_{j-1}(r)\|^{2}\right) \mathrm{d}s,
\end{align}
where $C_{1}$ and $C_{2}$ are defined in \eqref{constants}. 

Then, we have

\begin{align*}
\sup\limits_{1\leq j\leq k}\textbf{E}\left( \sup_{t\leq s\leq T}\|x_{j}(s)\|^{2}\right)&\leq C_{1}+C_{2}\int_{t}^{T}\sup\limits_{1\leq j\leq k}\textbf{E}\sup_{t\leq r\leq T}\left( \|x_{j-1}(r)\|^{2}\right)\mathrm{d}r.
\end{align*}
Applying Gronwall's inequality invokes that
\begin{align*}
\sup\limits_{1\leq j\leq k}\textbf{E}\left( \sup_{t\leq s\leq T}\|x_{j}(s)\|^{2}\right)&\leq C_{1}\exp(C_{2}(T-t)).
\end{align*}
Since $k$ was arbitrary, the inequality \eqref{11} follows. Finally it follows from \eqref{13}, we have
\begin{align*}
\textbf{E} \int_{t}^{T}\int_{s}^{T}\|y_{j}(s,u)\|^{2}\mathrm{d}u\mathrm{d}s&\leq \varpi^{-1}\left(  C_{1}+C_{2}\int_{s}^{T} C_{1}\exp(C_{2}(T-s))\mathrm{d}s\right) \\
&\leq \varpi^{-1} C_{1}\left(1+C_{2}(T-t)\exp(C_{2}(T-t)) \right).
\end{align*}
\end{proof}
\begin{lem}\label{lem3}
	Under Assumptions \ref{H1} and \ref{H3}, there exists a constant $C_{3}>0$ defined in \eqref{constants} such that
	
	\begin{align}\label{19}
	\textbf{E}\sup_{t\leq s\leq T}\|x_{j+k}(s)-x_{j}(s)\|^{2}\leq   C_{3}\int_{t}^{T}\rho\left(\textbf{E}\sup_{s\leq r\leq T}\|x_{j+k-1}(r)-x_{j-1}(r)\|^{2}\right)\mathrm{d}s,
	\end{align}
	for all $0\leq t \leq T$ and $j,k\geq 1$.
\end{lem}	
\begin{proof}
	Applying Lemma \ref{lem1}, we have 
	\allowdisplaybreaks
	\begin{align*}
	&\textbf{E}\sup_{t\leq s\leq T}\|x_{j+k}(s)-x_{j}(s)\|^{2}+\textbf{E} \int_{t}^{T}\int_{s}^{T}\|y_{j+k}(s,u)-y_{j}(s,u)\|^{2}\mathrm{d}u\mathrm{d}s\\
	&\leq 16\frac{(2T)^{2\alpha}}{2\alpha-1} \textbf{E}\int_{t}^{T}\int_{s}^{T}\|f(s,u,x_{j+k-1}(u),y_{j+k}(s,u))-f(s,u,x_{j-1}(u),y_{j}(s,u))\|^{2}\mathrm{d}u\mathrm{d}s\nonumber\\
	&+2\textbf{E} \int_{t}^{T}\int_{s}^{T}\|g(s,u,x_{j+k-1}(u))-g(s,u,x_{j-1}(u))\|^{2}\mathrm{d}u\mathrm{d}s\\
	&\leq 16\frac{(2T)^{2\alpha}}{2\alpha-1} \textbf{E}\int_{t}^{T}\int_{s}^{T}\left[ \rho\left( \|x_{j+k-1}(u)-x_{j-1}(u)\|^{2}\right)+c\|y_{j+k}(s,u)-y_{j}(s,u)\|^{2}\right]  \mathrm{d}u\mathrm{d}s\\
	&+2\textbf{E} \int_{t}^{T}\int_{s}^{T}\rho \left(\| x_{j+k-1}(u)-x_{j-1}(u)\|^{2}\right) \mathrm{d}u\mathrm{d}s\\
	&\leq 16 \frac{(2T)^{2\alpha}}{2\alpha-1}(T-t)\int_{t}^{T}\rho\left( \textbf{E}\sup_{s \leq r \leq T}\|x_{j+k-1}(r)-x_{j-1}(r)\|^{2}\right) \mathrm{d}s\\
	&+ 16 \frac{(2T)^{2\alpha}}{2\alpha-1}c\int_{t}^{T}\int_{s}^{T}\|y_{j+k}(s,u)-y_{j}(s,u)\|^{2}\mathrm{d}u\mathrm{d}s\\
	&+2(T-t)\textbf{E} \int_{t}^{T}\rho \left(\textbf{E}\sup_{s \leq r \leq T}\|x_{j+k-1}(r)-x_{j-1}(r)\|^{2}\right)\mathrm{d}s\\
	&\leq \left( 16 \frac{(2T)^{2\alpha}}{2\alpha-1}+2(T-t)\right)  \int_{t}^{T}\rho \left(\textbf{E}\sup_{s \leq r \leq T}\|x_{j+k-1}(r)-x_{j-1}(r)\|^{2}\right)\mathrm{d}s\\
	&+32 \frac{(2T)^{2\alpha}}{2\alpha-1}c\int_{t}^{T}\int_{s}^{T}\|y_{j+k}(s,u)-y_{j}(s,u)\|^{2}\mathrm{d}u\mathrm{d}s.
	\end{align*}
	Therefore, we have  
	\begin{align*}
	\textbf{E}\sup_{t\leq s\leq T}\|x_{j+k}(s)-x_{j}(s)\|^{2}&+\left(1- 32 \frac{(2T)^{2\alpha}}{2\alpha-1}c\right) \textbf{E} \int_{t}^{T}\int_{s}^{T}\|y_{j+k}(s,u)-y_{j}(s,u)\|^{2}\mathrm{d}u\mathrm{d}s\\
	&\leq C_{3}\int_{t}^{T}\rho \left(\textbf{E}\sup_{s \leq r \leq T}\|x_{j+k-1}(r)-x_{j-1}(r)\|^{2}\right)\mathrm{d}s,
	\end{align*}
	which completes the proof.
\end{proof}
\begin{lem}\label{lem4}
	Under assumptions \ref{H1}- \ref{H3}, there exists a constant $C_{4}>0$ defined in \eqref{constants} such that
	\begin{equation*}
	\textbf{E}\sup_{t\leq s\leq T}\|x_{j+k}(s)-x_{j}(s)\|^{2}\leq C_{4}(T-t),
	\end{equation*}
	for all $0\leq t \leq T$ and for all $j,k \geq 1$.
\end{lem}
\begin{proof}
	By Lemmas \ref{lem2} and \ref{lem3}, we have
	
	\begin{align*}
	\textbf{E}\sup_{t\leq s\leq T}\|x_{j+k}(s)-x_{j}(s)\|^{2}&\leq C_{3}\int_{t}^{T}\rho \left(\textbf{E}\sup_{s \leq r \leq T}\|x_{j+k-1}(r)-x_{j-1}(r)\|^{2}\right)\mathrm{d}s\\
	&\leq C_{3}\int_{t}^{T}\rho\left( 2C_{1}\exp(C_{2}(T-s))\right)\mathrm{d}s\\
	&\leq C_{3}\rho\left(2C_{1}\exp(C_{2}T)\right) (T-t)=C_{4}(T-t).
	\end{align*}
	Therefore, the proof is complete.
\end{proof}
Let us define the following sequences:
\begin{align*}
&\varphi_{1}(t)=C_{4}(T-t),\\
&\varphi_{j+1}(t)=C_{3}\int_{t}^{T}\rho( \varphi_{j}(s))\mathrm{d}s, \quad j\geq 1\\
&\tilde{\varphi}_{j,k}(t)=\textbf{E}\sup_{t\leq s\leq T}\|x_{j+k}(s)-x_{j}(s)\|^{2}, \quad  j\geq 1, k\geq 1,
\end{align*}
\begin{lem}\label{lem5}
	There exists $0\leq T_{0}\leq T$ such that for all $j,k\geq 1$
	\begin{equation*}
	0\leq \tilde{\varphi}_{j,k}(t)\leq \varphi_{j}(t)\leq \varphi_{j-1}(t)\leq \ldots \leq \varphi_{1}(t),  \quad \text{for all} \quad t\in [T_{0},T].
	\end{equation*}
\end{lem}
\begin{proof}
	We prove this lemma by mathematical induction principle in $j$.
	
	By Lemma \ref{lem4}, we have
	\begin{equation*}
	\tilde{\varphi}_{1,k}(t)=\textbf{E}\sup_{t\leq s\leq T}\|x_{1+k}(s)-x_{1}(s)\|^{2}\leq C_{4}(T-t)=\varphi_{1}(t).
	\end{equation*}
	By Lemma \ref{lem3}, we get
	\begin{align*}
	\tilde{\varphi}_{2,k}(t)&=\textbf{E}\sup_{t\leq s\leq T}\|x_{2+k}(s)-x_{2}(s)\|^{2}\\
	&\leq C_{3}\int_{t}^{T}\rho\left( \textbf{E}\sup_{s\leq r\leq T}\|x_{1+k}(r)-x_{1}(r)\|^{2}\right)\mathrm{d}s\\
	&=C_{3}\int_{t}^{T}\rho\left(\tilde{\varphi}_{1,k}(s)\right) \mathrm{d}s\leq C_{3}\int_{t}^{T}\rho\left(\varphi_{1}(s)\right) \mathrm{d}s=\varphi_{2}(t).
	\end{align*}
	We have to prove that there exists $T_{0}>0$ such that for all $t \in [T_{0},T]$ the following inequality holds:
	\begin{equation}\label{phi2}
	\varphi_{2}(t)=C_{3}\int_{t}^{T}\rho\left( C_{4}(T-s)\right) \mathrm{d}s\leq C_{4}(T-t)=\varphi_{1}(t).
	\end{equation}
	To this end, note that this inequality provided that
	
	\begin{equation*}
	C_{3}\rho\left(C_{4}(T-t)\right)\leq C_{4}=C_{3}\rho\left( 2C_{1}\exp(C_{2}T)\right)
	\end{equation*} 
	or 
	\begin{equation*}
	C_{3}\rho\left(2C_{1}\exp(C_{2}T)\right)\leq 2C_{1}\exp(C_{2}T)=2u
	\end{equation*} 
	On the other hand, this holds if 
	\begin{equation*}
	C_{3}(a+2bu)(T-t)\leq 2u
	\end{equation*}
	Since $u=C_{1}\exp(C_{2}T)\geq C_{1}$ the above inequality holds if 
	
	\begin{equation}\label{T-t}
	T-t\leq \frac{2}{C_{3}(\frac{a}{u}+2b)}
	\end{equation}
	Thus, \eqref{phi2} holds true for any $t$ satisfying \eqref{T-t}.
	Obviously, such a $t$ does not depend on the value $\xi$. Thus, there exists $T_{0}>0$ such that
	
	\begin{equation*}
	\varphi_{2}(t)\leq \varphi_{1}(t)
	\end{equation*}
	for all $t \in [T_{0},T]$. Now we assume that (20) holds for some $n\geq 2$. Then using the same inequalities as above yields
	\begin{align*}
	\tilde{\varphi}_{j+1,k}(t)&\leq C_{3}\int_{t}^{T}\rho\left(\textbf{E}\sup_{s\leq r\leq T}\|x_{j+k}(r)-x_{j}(r)\|^{2}\right) \mathrm{d}s\\
	&\leq C_{3}\int_{t}^{T}\rho\left(\tilde{\varphi}_{j,k}(s)\right) \mathrm{d}s\leq  C_{3}\int_{t}^{T}\rho\left(\varphi_{j}(s)\right) \mathrm{d}s=\varphi_{j+1}(t)
	\end{align*}
	for all $t\in [T_{0},T]$. On the other hand, we have 
	\begin{equation*}
	\varphi_{j+1}(t)=C_{3}\int_{t}^{T}\rho\left( \varphi_{j}(s)\right) \mathrm{d}s\leq C_{3}\int_{t}^{T}\rho\left( \varphi_{j-1}(s)\right) \mathrm{d}s=\varphi_{j}(t) \quad \text{for all}\quad t\in [T_{0},T].
	\end{equation*}
	This completes the proof.
\end{proof}	

\begin{thm}
	Assume that Assumptions \ref{H1}- \ref{H3} hold. Then there exists a unique mild solution $(x,y)$ of \eqref{fbsde}.
\end{thm}
\begin{proof}[Uniqueness:]
	To show the uniqueness let both $(x,y)$  and $(\tilde{x},\tilde{y})$ be solutions of \eqref{fbsde}. Then Lemma \ref{lem1} implies that
	\begin{align*}
	&\textbf{E}\sup_{t\leq s\leq T}\|x(s)-\bar{x}(s)\|^{2}+\textbf{E} \int_{t}^{T}\int_{s}^{T}\|y(s,u)-\bar{y}(s,u)\|^{2}\mathrm{d}u\mathrm{d}s\\
	&\leq 16\frac{(2T)^{2\alpha}}{2\alpha-1} \textbf{E}\int_{t}^{T}\int_{s}^{T}\|f(s,x(s),y(s,u))-f(s,\bar{x}(s),\bar{y}(s,u))\|^{2}\mathrm{d}s\nonumber\\
	&+2\textbf{E} \int_{t}^{T}\int_{s}^{T}\|g(s,u,x(s))-g(s,u,\bar{x}(s))\|^{2}\mathrm{d}u\mathrm{d}s\\
	&\leq \left( 16 \frac{(2T)^{2\alpha}}{2\alpha-1}+2(T-t)\right)  \int_{t}^{T}\rho \left(\textbf{E}\sup_{s \leq r \leq T}\|x(r)-\bar{x}(r)\|^{2}\right)\mathrm{d}s\\
	&+16\frac{(2T)^{2\alpha}}{2\alpha-1}c\textbf{E} \int_{t}^{T}\int_{s}^{T}\|y(s,u)-\bar{y}(s,u)\|^{2}\mathrm{d}u\mathrm{d}s.
	\end{align*}
	For $t\leq s\leq T$, we have
	\begin{align*}
	\textbf{E}\sup_{t\leq s\leq T}\|x(s)-\bar{x}(s)\|^{2}+&\textbf{E} \int_{t}^{T}\int_{s}^{T}\|y(s,u)-\bar{y}(s,u)\|^{2}\mathrm{d}u\mathrm{d}s\leq C\int_{t}^{T}\rho\left(\textbf{E}\sup_{s\leq r\leq T}\|x(r)-\bar{x}(r)\|^{2} \right)\mathrm{d}s\\
	&+8\frac{(2T)^{2\alpha}}{2\alpha-1}c\textbf{E} \int_{t}^{T}\int_{s}^{T}\|y(s,u)-\bar{y}(s,u)\|^{2}\mathrm{d}u\mathrm{d}s.
	\end{align*}
	Let $1-8\frac{(2T)^{2\alpha}}{2\alpha-1}c>0$, then for any $0\leq t \leq T$, we have
	\begin{equation*}
	\textbf{E}\sup_{t\leq s\leq T}\|x(s)-\bar{x}(s)\|^{2}\leq C\int_{t}^{T}\rho\left(\textbf{E}\sup_{s\leq r\leq T}\|x(r)-\bar{x}(r)\|^{2} \right)\mathrm{d}s.
	\end{equation*}
	
	Therefore, by Bihari's inequality we obtain
	\begin{equation*}
	\textbf{E}\sup_{t\leq s\leq T}\|x(s)-\tilde{x}(s)\|^{2}=0.
	\end{equation*}
	So $x(t)=\bar{x}(t)$ for all $0\leq t\leq T$ almost surely. It then follows from \eqref{ineq1} that $y(t,s)=\bar{y}(t,s)$ for all $(t,s)\in \mathcal{D}$ almost surely as well. This establishes the uniqueness.
	
	\textit{Existence:} We claim that
	\begin{equation}\label{goesto0}
	\textbf{E}\sup_{t\leq s\leq T}\|x_{j+k}(s)-x_{j}(s)\|^{2}\to 0,\quad \text{for all} \quad T_{0}\leq t\leq T, \quad \text{as} \quad j,k\to \infty.
	\end{equation}
	Note that by definition , $\varphi_{j}$ is continuous on $[T_{0},T]$ and also for each $n\geq 1$, $\varphi_{j}(\cdot)$ is decreasing on $[T_{0},T]$ and for each $\varphi_{j}(t)$ is a nonincreasing sequence. Therefore, we can define the function $\varphi(t)$ by $\varphi_{j}(t)\downarrow \varphi(t)$. It is easy to verify that $\varphi(t)$ is continuous and nonincreasing on $[T_{0},T]$. By definition of $\varphi_{j}(t)$ and $\varphi$(t), we get
	\begin{equation*}
	\varphi(t)= \lim\limits_{j\to \infty}C_{3}\int_{t}^{T}\rho(\varphi_{j}(s))\mathrm{d}s=C_{3}\int_{t}^{T}\rho(\varphi(s))\mathrm{d}s
	\end{equation*}
	for each $t\in [T_{0},T]$. Since
	\begin{equation*}
	\int_{0+}\frac{\mathrm{d}u}{\rho(u)}=\infty,
	\end{equation*}
	By virtue of Assumption \ref{H3}, $\varphi(t)=0$ for all  $t\in [T_{0},T]$. As a consequence, $\lim\limits_{j\to \infty}\varphi_{j}(T_{0})=0$. By Lemma \ref{lem5}
	\begin{align*}
	\textbf{E}\sup_{t\leq s\leq T}\|x_{j+k}(s)-x_{j}(s)\|^{2}&\leq \sup_{T_{0}\leq t\leq T}\tilde{\varphi}_{j,k}(t)\\
	&\leq \sup_{T_{0}\leq t\leq T}\tilde{\varphi}_{j}(t)=\varphi_{j}(T_{0})\to 0, \quad \text{as} \quad j\to \infty.
	\end{align*}
	So \eqref{goesto0} must hold. Applying \eqref{goesto0} to \eqref{19}, we see that $\left\lbrace x_{j},y_{j} \right\rbrace $ is a Cauchy sequence (hence convergent) in $M[T_{0},T]$ and its limit is denoted by $(x,y)$. Now letting $j\to \infty$ in \eqref{9}, we obtain 
	\begin{align*}
	x(t)=\xi &+\int_{t}^{T}(s-t)^{\alpha-1}f(t,s,x(s),y(t,s))\mathrm{d}s\\
	&+\int_{t}^{T}(s-t)^{\alpha-1}\left[g(t,s,x(s))+y(t,s) \right] \mathrm{d}w(s),
	\end{align*}
	on $[T_{0},T]$. Since the value of $T_{0}$ depends only on the function $\rho$, one can deduce by iteration the existence on $[T-q(T-T_{0}),T]$ for each $q$, and therefore the existence on the entire interval $[0,T]$.
\end{proof}

\section{Conclusions and future works}\label{open problems}
In this paper, we first formulated new problem in BSDE theory which is singular backward stochastic nonlinear Volterra differential equation, is an untreated topic in recent literature. To derive an adapted pair of stochastic processes, we first formulated fundamental lemma which plays a crucial role in the theory of singular BSDE. The main results in our paper were to show  existence and uniqueness of an adapted solution to \eqref{fbsde} in infinite dimensional setting using  Carath\'{e}odory-type condition. In doing so, we constructed Picard type approximation. The key point in the proof of main results was to apply extended martingale representation theorem and Bihari's inequality. 

Since our results are sufficiently new in the theory of BSDEs, there are still open problems to discuss regarding their applications to finance and optimal control theory using stochastic maximum principle.
\section*{Acknowledgment}
On behalf of all authors, the corresponding author states that there is no conflict of interest.

\end{document}